\documentclass[12pt,a4paper]{article}
\usepackage[a4paper,top=2cm,bottom=2cm,left=2cm,right=2cm,bindingoffset=5mm]{geometry}
\usepackage[english
]{babel}
\usepackage{amsfonts}
\usepackage{indentfirst}
\usepackage{color}
\usepackage{graphicx}

\usepackage{amssymb}
\usepackage{amsmath}
\usepackage{latexsym}
\usepackage{amsthm}
\usepackage{mathrsfs}
\usepackage{hyperref}
\hypersetup{colorlinks=true,linkcolor=blue}

\theoremstyle{plain}

\newtheorem{teo}{Theorem}[section]

\newtheorem{lemma}[teo]{Lemma}
\newtheorem{prop}[teo]{Proposition}
\newtheorem{corol}[teo]{Corollary}

\theoremstyle{definition}

\theoremstyle{remark}
\newtheorem{oss}[teo]{Remark}
{\left\lbrace\begin{array}{@{}l@{}}}%
{\end{array}\right.}

\newcommand{\R}{\mathbb{R}}

\newcommand{\N}{\mathbb{N}}
\newcommand{\elle}{\mathcal{L}}

\DeclareMathOperator{\support}{supp}

\date{}
\begin{document}
\selectlanguage{english}%

\title{\bf{Regularity results for nonlocal evolution Venttsel' problems}}
\author{Simone Creo\thanks{Dipartimento di Scienze di Base e Applicate per l'Ingegneria, Sapienza Universit\`{a} di Roma, Via A. Scarpa 16, 00161 Roma, Italy.
   E-mail: simone.creo@sbai.uniroma1.it, mariarosaria.lancia@sbai.uniroma1.it},\setcounter{footnote}{6}
Maria Rosaria Lancia$^*$, Alexander Nazarov\thanks{St.~Petersburg Department of Steklov Mathematical Institute, Fontanka 27, 191023 St.~Petersburg, Russia, and St.~Petersburg State University, Universitetskii pr. 28, 198504 St.~Petersburg, Russia. E-mail: al.il.nazarov@gmail.com}
		}
\maketitle

\begin{abstract}
\noindent We consider parabolic nonlocal Venttsel' problems in polygonal and piecewise smooth two-dimensional domains and study existence, uniqueness and regularity in (anisotropic) weighted Sobolev spaces of the solution.
\end{abstract}

\bigskip

\noindent\textbf{Keywords: }Venttsel' problems, nonlocal operators, anisotropic weighted Sobolev spaces, piecewise smooth domains.\\

\noindent{\textbf{AMS Subject Classification:} 35K20, 35B65. Secondary: 35R02, 35B45.}

 \pagestyle{myheadings} \thispagestyle{plain}

\section*{Introduction}

\noindent The aim of this paper is to study the heat equation with nonlocal Venttsel' boundary conditions in a bounded polygonal domain $\Omega\subset\R^2$. In the cornerstone paper of Venttsel' \cite{vent59}, a nonlocal term already appears; only recently, many papers deal with nonlocal Venttsel' problems both in the case of smooth and irregular domains. Among the others, we refer to \cite{LVSV}, \cite{VELEZ2015}, \cite{WAR12}, \cite{ambprodi} and the references listed in. In this paper, we consider a nonlocal term which can be regarded as a suitable version of the regional fractional Laplacian $(-\Delta)^s$, for $s\in(0,1)$, see \cite{nostromassimo} for applications.

Actually, Venttsel' problems in irregular domains (in particular, domains with pre-fractal or fractal boundary) have been widely investigated, see, e.g., \cite{JEE,LRDV}, where the reader can find also the motivations. We refer, for local linear and quasi-linear Venttsel' problems, to \cite{AP-NAsurv}, \cite{AP-NA1}, \cite{AP-NA2}, \cite{nazarov}, \cite{Ar-Me-Pa-Ro}, \cite{Gol-Ru}, \cite{warma2009}, \cite{VELEZ2014}, \cite{CPAA}, \cite{miovalerio}, \cite{nostroJEE}, \cite{ApNazPalSof} and the references listed in.

In this paper, our goal is to prove regularity results in weighted Sobolev spaces for the weak solution of the problem at hand, thus extending the results obtained in \cite{nostronazarov} for the elliptic case. When considering the numerical approximation of this problem, to prove regularity results is a key issue for obtaining optimal a priori error estimates. To this regard, see \cite{Ce-Dl-La,Ce-La-Hd} for the local case, and \cite{nostromassimo} for the nonlocal case, under stronger assumptions on the data.

As in the elliptic case \cite{nostronazarov}, it is crucial to prove that the weak solution of the nonlocal Venttsel' problem belongs for a.e. $t$ to the space $H^2(\partial\Omega)$; this is achieved by the so-called \emph{Munchhausen trick}, see, e.g., \cite{nostronazarov}, \cite{ApNazPalSof}. To this aim, we introduce suitable anisotropic weighted Sobolev spaces of Kondrat'ev type, see \cite{kond,kozlov}, where the weight is the distance from the set of vertices. The techniques used to prove the regularity on the boundary, in the parabolic case, deeply rely also on sophisticated extension theorem in anisotropic Sobolev spaces.

The paper is organized as follows. In Section \ref{sec1} we define the domain and the functional spaces appearing in this paper, and state the problem. In Section \ref{sec2} we prove a crucial a priori estimate for the solution. In Section \ref{sec3} we give an existence and uniqueness result for weak and strong solutions of the parabolic nonlocal Venttsel' problem. 
Appendix \ref{appendice} contains the extension theorem from the broken surface to the whole space.

\section{Statement of the problem}\label{sec1}
\setcounter{equation}{0}

\noindent Let $\Omega\subset\R^2$ be a domain with polygonal boundary $\partial\Omega$ with vertices $V_j$, for $j=1,\dots,N$. Namely, we suppose that $\partial\Omega$ is made by $N\geq 3$ segments $l_j$, which form a finite number of angles with opening $\alpha_j$, and let us denote with $\alpha$ the opening of the largest angle in $\partial\Omega$, see Figure \ref{esempiodominio}. We denote by $V(\partial\Omega)$ the set of vertices $V_j$.

\begin{figure}[htp]
\centering
\includegraphics[width=0.6\textwidth]{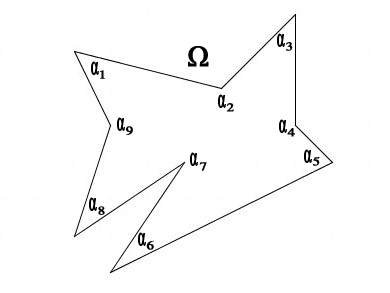}%
\caption{A possible example of domain $\Omega$. In this case $N=9$ and $\alpha=\alpha_7$.}%
\label{esempiodominio}
\end{figure}

\noindent In the following we denote with $L^2(\Omega)$ the Lebesgue space with respect to the Lebesgue measure $dx$ on $\Omega$, and with $L^2(\partial\Omega)$ the Lebesgue space on the boundary with respect to the arc length $d\ell$. By $H^s(\Omega)$, for $s>0$, we denote the standard Sobolev--Slobodetskii spaces. By $\mathcal{C}(\partial\Omega)$ we denote the set of continuous functions on $\partial\Omega$, and by $\mathcal{C}^\infty_0(\R\times\R)$ we denote the set of infinitely differentiable functions with compact support in $\R\times\R$. Moreover, we denote by $B_1(0)$ the unit ball centered in the origin.

\noindent By $H^s(\partial\Omega)$, for $0<s<1$, we denote the Sobolev--Slobodetskii space on $\partial\Omega$ defined by local Lipschitz charts as in~\cite{necas}. For $s\geq 1$, we define the space $H^s(\partial\Omega)$ by using the characterization given by Brezzi-Gilardi in~\cite{bregil}:
\begin{equation}\notag
H^s (\partial\Omega)=\{v\in \mathcal{C} (\partial\Omega)\,:\, v|_{\overset{\circ}{M}}\in H^s (\overset{\circ}{M})\},
\end{equation}
where $M$ denotes a side of $\partial\Omega$ and $\overset{\circ}{M}$ denotes the corresponding open segment (for the general case see Definition 2.27 in \cite{bregil}).

\noindent We fix a counterclockwise orientation on $\partial\Omega$. We denote by $L_j$ the length of the segment $l_j$, for $j=1,\dots,N$, and by $L$ the length of $\partial\Omega$.\\
We choose $V_1$ as the origin and we parametrize $\partial\Omega$ by the arc-length
\begin{equation}\notag
\phi_1(\ell)=\phi_{V_1}(\ell),\quad\phi_1(\ell)\colon[0,L]\to\R^2,
\end{equation}
with $\phi_1$ continuous, injective in $[0,L]$ and such that $\phi_1(0)=\phi_1(L)$.\\
By choosing as origin $V_j$, we define in a similar way
\begin{equation}\notag
\phi_j(\ell)=\phi_{V_j}(\ell),\quad\ell\in[0,L].
\end{equation}
For every $j=1,\dots,N$, we set 
\begin{equation*}
\nabla_\ell^+u(V_j):=\lim_{h\to 0^+}\nabla_\ell u(\phi_j(L_j+h)),\qquad\nabla_\ell^-u(V_j):=\lim_{h\to 0^-}\nabla_\ell u(\phi_j(L_j+h))\,,
\end{equation*}
where $\displaystyle\nabla_\ell=\frac{\partial}{\partial\ell}$,
and define the subspace
\begin{equation*}
\tilde H^2(\partial\Omega)=\{u\in H^2(\partial\Omega)\,:\,u\circ\phi_j\in H^2(0,L_j)\text{ and }\nabla_\ell^+u(V_j)=\nabla_\ell^-u(V_j)\,\,\forall\,j=1,\dots,N\}.
\end{equation*}

\noindent Let $r=r(x)$ be the distance from the set of vertices $V_j$. For $\gamma\in\R$, and $m=0,1,2,\dots$, we denote by $H^m_\gamma (\Omega)$ the Kondrat'ev space of functions for which the norm
\begin{equation}\notag
\|u\|_{H^m_\gamma (\Omega)}=\left(\sum_{|k|\leq m} \int_{\Omega} r^{2(\gamma-m+|k|)} |D^{k} u(x)|^2\,dx\right)^{\frac{1}{2}}
\end{equation}
is finite, see \cite{kond}. For $m=0$, this space evidently coincides with the weighted Lebesgue space $L^2_\gamma (\Omega)$. We also define, for $m\in\N$, the space $H^{m-\frac{1}{2}}_\gamma (\partial\Omega)$ as the trace space of $H^m_\gamma (\Omega)$ equipped with the norm
\begin{equation}\notag
\|u\|_{H^{m-\frac{1}{2}}_\gamma (\partial\Omega)}=\inf_{v=u\,\text{on}\,\partial\Omega}\,\|v\|_{H^m_\gamma (\Omega)}.
\end{equation}

\noindent We now introduce anisotropic Sobolev spaces on the cylinder $Q_T=\Omega\times(0,T)$ and its lateral surface $\partial'' Q_T=\partial\Omega\times(0,T)$. For $l,m\geq 0$ we define
\begin{equation}\notag
H^{l,m}(Q_T)=L^2([0,T];H^l(\Omega))\cap H^m([0,T];L^2(\Omega)),
\end{equation}
and by $H^{l,m}(\partial''Q_T)$ we denote the analogous space on $\partial''Q_T$, taking into account the previous definition of the space $H^s(\partial\Omega)$. 
Similarly, we define
\begin{equation}\notag
\tilde H^{2,1} (\partial''Q_T)=L^2([0,T];\tilde H^2(\partial\Omega))\cap H^1([0,T];L^2(\partial\Omega)).
\end{equation}

\noindent We introduce also the anisotropic Kondrat'ev space $H^{2,1}_\gamma (Q_T)$ of functions for which the following norm is finite (see \cite{kozlov}):
\begin{equation}\notag
\|u\|_{H^{2,1}_\gamma (Q_T)}=\left(\int_{Q_T} r^{2(\gamma-2)}\sum_{|\bar\alpha|\leq 2} r^{2|\bar\alpha|} |\partial_t^{\alpha_0} D_x^{\alpha} u|^2\,dxdt\right)^{\frac{1}{2}},
\end{equation}
where $\bar\alpha=(\alpha_0,\alpha)$ and $|\bar\alpha|=2\alpha_0+|\alpha|$.

\noindent We denote the trace of $u$ on $\partial''Q_T$ with $\gamma_0 u$. Sometimes we will use the same symbol $u$ to denote the function itself and its trace $\gamma_0 u$. The interpretation will be left to the context.\\
We define the composite spaces
$$V^{1,0}(Q_T,\partial''Q_T):=\{u\in H^{1,0} (Q_T)\,:\, \gamma_0 u\in H^{1,0} (\partial''Q_T)\},$$
$$V^{1,1}(Q_T,\partial''Q_T):=\{u\in H^{1,1} (Q_T)\,:\, \gamma_0 u\in H^{1,1} (\partial''Q_T)\},$$
and, for $\sigma\in\R$,
$$V^{2,1}_\sigma(Q_T,\partial''Q_T):=\{u\in H^{1,0} (Q_T)\,:\, r^\sigma D^2 u\in L^2(Q_T),\,r^\sigma u_t\in L^2(Q_T),\,\gamma_0 u\in \tilde H^{2,1}(\partial''Q_T)\}.$$

\noindent We consider the problem formally stated as
\begin{align}
&u_t-\Delta u+au=f & &\text{in $Q_T$,}\label{pbform1}\\[2mm]
&u_t-\Delta_\ell u+\frac{\partial u}{\partial\nu}+bu+\theta_s(u)=g & &\text{on $\partial''Q_T$}\label{pbform2},\\[2mm]
&u(x,0)=0 & &\text{ on $\overline\Omega$}\label{pbform3},
\end{align}
where $f$ and $g$ are given functions, $\displaystyle\Delta_\ell=\frac{\partial^2}{\partial\ell^2}$, $\nu$ is the unit vector of exterior normal, $a\in L^\infty (Q_T)$, $b\in L^\infty (\partial''Q_T)$ and, for $s\in(0,1)$, we set $\theta_s\colon H^{s} (\partial\Omega)\to H^{-s} (\partial\Omega)$ as follows: for every $u,v\in H^{s} (\partial\Omega)$
\begin{equation}\notag
\langle\theta_s (u),v\rangle=\iint_{\partial\Omega\times \partial\Omega}\frac{(u(x)-u(y))(v(x)-v(y))}{|x-y|^{1+2s}}\,d\ell(x)\,d\ell(y),
\end{equation}
where $\langle\cdot,\cdot\rangle$ denotes the duality pairing between $H^{-s} (\partial\Omega)$ and $H^{s} (\partial\Omega)$. We remark that the nonlocal term $\theta_s(\cdot)$ can be regarded as an analogue of the regional fractional Laplace operator $(-\Delta)_{\partial\Omega}^s$ on $\partial\Omega$.\\
We now define the bilinear form $E(u,v)$ as follows:
\begin{equation}
E(u,v)=\int_{\Omega} \nabla u\,\nabla v\,dx+\int_{\partial\Omega} \nabla_\ell u\,\nabla_\ell v\,d\ell+\int_{\Omega} a\,u\,v\,dx+\int_{\partial\Omega} b\,u\,v\,d\ell+\langle\theta_s (u),v\rangle,
\end{equation}
for every $u,v\in V^1(\Omega,\partial\Omega):=\{u\in H^1 (\Omega)\,:\, u|_{\partial\Omega}\in H^1 (\partial\Omega)\}$.\\
\noindent We consider the weak formulation of the problem \eqref{pbform1}-\eqref{pbform3} (cf. \cite{LiMa1}):
\begin{equation}\label{deb}
\begin{split}
&\text{Given $f$ and $g$, find $u\in V^{1,0}(Q_T,\partial'' Q_T)$ such that}\\[2mm]
\displaystyle &-\int_{Q_T} u\,v_t\,dx\,dt-\int_{\partial''Q_T} u\,v_t\,d\ell\,dt+\int_0^T E(u,v)\,dt=\int_{Q_T} f\,v\,dx\,dt+\int_{\partial'' Q_T} g\,v\,d\ell\,dt\\[2mm]
&\text{for every $v\in V^{1,1}(Q_T,\partial'' Q_T)$ such that $v(T,x)=0$.}
\end{split}
\end{equation}

\begin{prop}\label{tilde}
 Let $u$ be a weak solution of \eqref{pbform1}-\eqref{pbform3}. Suppose that $r^\sigma D^2 u\in L^2(Q_T)$, $r^\sigma u_t\in L^2(Q_T)$ and $\gamma_0 u\in H^{2,1}(\partial''Q_T)$. Then $u$ is a strong solution, i.e. equalities \eqref{pbform1}-\eqref{pbform3} are satisfied a.e. in $Q_T$, on $\partial''Q_T$ and in $\Omega$, respectively. Moreover, $\gamma_0 u\in \tilde H^{2,1}(\partial''Q_T)$, i.e. $u\in V^{2,1}_\sigma(Q_T,\partial''Q_T)$.
\end{prop}
This statement follows from integration by parts and the fundamental lemma of calculus of variations.

\medskip

In what follows we denote by $C$ all positive constants. The dependence of constants on some parameters is given in parentheses. We do not indicate the dependence of $C$ on the geometry of $\Omega$.

\section{A priori estimates}\label{sec2}
\setcounter{equation}{0}

\begin{teo}\label{aprioriest} Let $u\in V^{2,1}_\sigma(Q_T,\partial''Q_T)$ be a solution of problem \eqref{pbform1}-\eqref{pbform3}. Then there exists a positive constant $C=C(\sigma)$ such that 
\begin{equation}\label{stima5}
\begin{split}
&\|u\|^2_{H^{1,0}(Q_T)}+\|r^\sigma D^2 u\|^2_{L^2(Q_T)}+\|r^\sigma u_t\|^2_{L^2(Q_T)}+\|u\|^2_{H^{2,1}(\partial''Q_T)}\\[3mm]
&\leq C(\sigma)\left(\|u\|^2_{L^2(\partial''Q_T)}+\|r^\sigma f\|^2_{L^2(Q_T)}+\|g\|^2_{L^2(\partial''Q_T)}\right),
\end{split}
\end{equation}
provided 
\begin{equation}\label{boundsigma}
1-\frac{\pi}{\alpha}<\sigma<\frac{1}{2},\qquad \sigma\geq-\frac{1}{2}
\end{equation}
(recall that $\alpha$ is the opening of the largest angle in $\partial\Omega$).
\end{teo}

\begin{proof}
We use the \emph{Munchhausen trick}. 
We move the terms $\frac{\partial u}{\partial\nu}$, $bu$ and $\theta_s(u)$ in \eqref{pbform2} into the right-hand side and consider them as known functions. Then we easily have 
\begin{equation}\label{stima1}
\|u\|^{2}_{H^{2,1}(\partial''Q_T)}\leq C\left(\left\|\frac{\partial u}{\partial\nu}\right\|^2_{L^2(\partial''Q_T)} + \|u\|^2_{L^2(\partial''Q_T)}+\|\theta_s (u)\|^2_{L^2(\partial''Q_T)}+\|g\|^2_{L^2(\partial''Q_T)}\right).
\end{equation}
We proceed in several steps.

{\bf 1)} First we estimate $\|\theta_s(u)\|^2_{L^2(\partial''Q_T)}$. Since $u\in\tilde H^{2,1}(\partial''Q_T)$, in particular $u(\cdot,t)\in\tilde H^2(\partial\Omega)$ for a.e. $t$. Hence it is sufficient to consider the local behavior of $u$ near the vertices. Without loss of generality, we can assume that the vertex is located at the origin. We introduce a smooth cutoff function $\eta$ and rectify $\partial\Omega$ near the origin. From our hypothesis on $u$, we have that for a.e. $t\in[0,T]$ $\theta_s(u(\cdot,t))\in H^{2-2s}(\partial\Omega)$ and
\begin{equation}\notag
\|\theta_s(u(\cdot,t))\|^2_{H^{2-2s}(\partial\Omega)}\leq C(s)\|u(\cdot,t)\|^2_{H^2(\partial\Omega)}.
\end{equation}
From the compact embedding of $H^{2-2s} (\partial\Omega)$ in $L^2(\partial\Omega)$ we deduce that for every $\varepsilon>0$ there exists a constant $C(\varepsilon)$ such that
\begin{equation}\notag
\|\theta_s(u(\cdot,t))\|^2_{L^2(\partial\Omega)}\leq\varepsilon\|\theta_s(u(\cdot,t))\|^2_{H^{2-2s} (\partial\Omega)}+C(\varepsilon)\|\theta_s(u(\cdot,t))\|^2_{H^{-s}(\partial\Omega)},
\end{equation}
see Lemma 6.1, Chapter 2 in~\cite{necas}. Similarly, we have
\begin{equation}\notag
\|\theta_s(u(\cdot,t))\|^2_{H^{-s}(\partial\Omega)}\leq C\|u(\cdot,t)\|^2_{H^s(\partial\Omega)}\leq\varepsilon\|u(\cdot,t)\|^2_{H^2(\partial\Omega)}+C(\varepsilon)\|u\|^2_{L^2(\partial\Omega)}.
\end{equation}
Putting together these estimates, we get
\begin{equation}\label{spaziale}
\|\theta_s(u(\cdot,t))\|^2_{L^2(\partial\Omega)}\leq\varepsilon\|u(\cdot,t)\|^2_{H^{2} (\partial\Omega)}+C(\varepsilon)\|u(\cdot,t)\|^2_{L^2(\partial\Omega)}.
\end{equation}
By integrating \eqref{spaziale} with respect to $t\in[0,T]$ we obtain
\begin{equation}\label{spaztemp}
\begin{split}
\|\theta_s(u)\|^2_{L^2(\partial''Q_T)}=\int_0^T\|\theta_s(u(\cdot,t))\|^2_{L^2(\partial\Omega)}\,dt\leq\int_0^T\left(\varepsilon\|u(\cdot,t)\|^2_{H^{2} (\partial\Omega)}+C(\varepsilon)\|u(\cdot,t)\|^2_{L^2(\partial\Omega)}\right)\,dt\\[2mm]
=\varepsilon\|u\|^2_{L^2([0,T];H^2(\partial\Omega))}+C(\varepsilon)\|u\|^2_{L^2(\partial''Q_T)}\leq\varepsilon\|u\|^2_{H^{2,1}(\partial''Q_T)}+C(\varepsilon)\|u\|^2_{L^2(\partial''Q_T)}.
\end{split}
\end{equation}

\noindent Therefore we obtain the following estimate using \eqref{stima1}:
\begin{equation}\notag
\|u\|^2_{H^{2,1} (\partial''Q_T)} \leq C\left(\left\|\frac{\partial u}{\partial\nu}\right\|^2_{L^2(\partial''Q_T)} +\|g\|^2_{L^2(\partial''Q_T)}+\varepsilon\|u\|^2_{H^{2,1}(\partial''Q_T)}+C(\varepsilon)\|u\|^2_{L^2(\partial''Q_T)}\right).
\end{equation}
By choosing $\varepsilon$ sufficiently small we obtain
\begin{equation}\label{intermedia}
\|u\|^2_{H^{2,1}(\partial''Q_T)}\leq C\left(\left\|\frac{\partial u}{\partial\nu}\right\|^2_{L^2(\partial''Q_T)} + \|u\|^2_{L^2(\partial''Q_T)}+\|g\|^2_{L^2(\partial''Q_T)}\right).
\end{equation}

{\bf 2)} By Theorem \ref{ext}, there is an extension $U\in H^{\frac{5}{2},\frac{5}{4}}(\R^2\times\R)$ such that $(U-u)|_{\partial''Q_T}=0$, $U|_{t=0}=0$, and the following estimate holds:
\begin{equation}\label{stimaU}
\|U\|_{H^{\frac{5}{2},\frac{5}{4}}(\R^2\times\R)}\leq C\|u\|_{H^{2,1}(\partial''Q_T)}.
\end{equation}

Without loss of generality we can suppose that the support of $U$ is bounded.

We claim that $D^2 U$ and $U_t$ belong to the weighted Lebesgue space $L^2_{-\frac{1}{2}}(\R^2\times\R)$. Indeed, by localizing we need to check it only in a neighborhood of a vertex $V_j$ located at the origin.

The inclusion of $U\in H^{\frac{5}{2},\frac{5}{4}}(\R^2\times\R)$ evidently implies $D^2U\in L^2(\R;H^\frac{1}{2}(\R^2))$. Furthermore, the Young inequality
\begin{equation*}
|\eta||\xi|^\frac{1}{2}\leq\frac{|\eta|^\frac{5}{4}}{5/4}+\frac{|\xi|^\frac{5}{2}}{5}
\end{equation*}
shows that $U_t\in L^2(\R;H^\frac{1}{2}(\R^2))$, and \eqref{stimaU} gives
\begin{equation}\label{stimaU2}
\|D^2 U\|_{L^2(\R;H^\frac{1}{2}(\R^2))}+\|U_t\|_{L^2(\R;H^\frac{1}{2}(\R^2))}\leq C\|u\|_{H^{2,1}(\partial''Q_T)}.
\end{equation}

By the fractional Hardy inequality, see \cite[Theorem 3.2 and Remark 3.2]{ilin}, for a.e. $t$ we have 
\begin{equation}\label{stimaU3}
\int_{\R^2}\frac{|D^2U(\cdot,t)|^2}{|x|}\,dx\leq C\|D^2 U(\cdot,t)\|^2_{H^\frac{1}{2}(\R^2)},
\end{equation}
and a similar inequality holds for $U_t$. We integrate these estimates with respect to $t$, and the claim follows.
\medskip

{\bf 3)} We now consider the function $v=u-U$. It solves the Dirichlet problem
\begin{equation}\label{problemav}
v_t-\Delta v=f-U_t+\Delta U\in L^2_\sigma (Q_T);\qquad v|_{\partial''Q_T}=0;\qquad v|_{t=0}=0. 
\end{equation}
(here we used the last restriction in \eqref{boundsigma}). From Theorem 3 in~\cite{kozlovmazya} (with $l=0$) it follows that $v\in H^{2,1}_{\sigma} (Q_T)$ if $\displaystyle|\sigma-1|<\frac{\pi}{\alpha}$ (we recall that $\alpha$ is the opening of the largest angle in $\partial\Omega$).

From \eqref{stimaU2} and \eqref{stimaU3}, this implies
\begin{equation}\label{stimaU+v}
\|u\|^2_{H^{1,0}(Q_T)}+\|r^\sigma D^2 u\|^2_{L^2(Q_T)}+\|r^\sigma u_t\|^2_{L^2(Q_T)}\leq C(\sigma)(\|r^\sigma f\|^2_{L^2(Q_T)}+\|u\|^2_{H^{2,1}(\partial''Q_T)})
\end{equation}
(to estimate the first term, we also take into account that \eqref{boundsigma} implies $\sigma\leq 1$).

\medskip

{\bf 4)} We are now in the position to estimate $\left\|\frac{\partial u}{\partial\nu}\right\|^2_{L^2(\partial''Q_T)}$. By rescaling, we deduce that $\nabla u\in L^2_{\sigma-\frac{1}{2}} (\partial''Q_T)$ and
\begin{equation}\label{stimaderivata}
\|\nabla u\|^2_{L^2_{\sigma-\frac{1}{2}} (\partial''Q_T)}\leq C\left(\|u\|^2_{H^{1,0}(Q_T)}+\|r^\sigma D^2 u\|^2_{L^2(Q_T)}\right).
\end{equation}

Following \cite{nostronazarov}, we define a cutoff function $\eta_\delta$ such that
\begin{equation}\notag
\eta_\delta(r)=1\quad\text{for}\quad r>\delta,\qquad \eta_\delta (r)=0\quad\text{for}\quad r<\delta/2
\end{equation}
and we introduce the following trace operator:
\begin{equation}\notag
u\longrightarrow\frac{\partial u}{\partial\nu}\Big|_{\partial''Q_T}=\eta_\delta\frac{\partial u}{\partial\nu}\Big|_{\partial''Q_T}+(1-\eta_\delta)\frac{\partial u}{\partial\nu}\Big|_{\partial''Q_T}=:\mathcal{K}_1(\delta)u+\mathcal{K}_2(\delta)u.
\end{equation}
The operator $\mathcal{K}_1(\delta)\colon H^{2,1}_\sigma(Q_T)\to L^2(\partial''Q_T)$ is evidently compact. Using \eqref{stimaU+v}, we obtain for arbitrary $\varepsilon>0$
\begin{equation}\notag
\|\mathcal{K}_1(\delta)u\|^2_{L^2(\partial''Q_T)}\leq\frac{\varepsilon}{2} (\|r^\sigma f\|^2_{L^2(Q_T)}+\|u\|^2_{H^{2,1}(\partial''Q_T)})+C(\varepsilon,\sigma,\delta)\|u\|^2_{L^2(\partial''Q_T)}.
\end{equation}

From \eqref{stimaU+v} and \eqref{stimaderivata} we deduce
\begin{equation}\notag
\|\mathcal{K}_2(\delta)u\|^2_{L^2(\partial''Q_T)}\leq C(\sigma)\delta^{\frac{1}{2}-\sigma}(\|r^\sigma f\|^2_{L^2(Q_T)}+\|u\|^2_{H^{2,1}(\partial''Q_T)}).
\end{equation}
By choosing $\delta(\sigma,\varepsilon)$ sufficiently small we get
\begin{equation}\notag
\left\|\frac{\partial u}{\partial\nu}\right\|^2_{L^2(\partial''Q_T)}\leq\varepsilon(\|r^\sigma f\|^2_{L^2(Q_T)}+\|u\|^2_{H^{2,1}(\partial''Q_T)})+C(\varepsilon,\sigma)\|u\|^2_{L^2(\partial''Q_T)}.
\end{equation}

\noindent Substituting the above inequality into \eqref{intermedia} we have
\begin{equation}\notag
\|u\|^2_{H^{2,1}(\partial''Q_T)}\leq C\left(\varepsilon(\|r^\sigma f\|^2_{L^2(Q_T)}+\|u\|^2_{H^{2,1}(\partial''Q_T)})+C(\varepsilon,\sigma)\|u\|^2_{L^2(\partial''Q_T)}+\|g\|^2_{L^2(\partial''Q_T)}\right).
\end{equation}
By choosing $\varepsilon$ sufficiently small we obtain
\begin{equation*}\label{intermedia2}
\|u\|^2_{H^{2,1}(\partial''Q_T)}\leq C\left(\|r^\sigma f\|^2_{L^2(Q_T)}+C(\sigma)\|u\|^2_{L^2(\partial''Q_T)}+\|g\|^2_{L^2(\partial''Q_T)}\right).
\end{equation*}
Taking into account \eqref{stimaU+v}, we get \eqref{stima5}.
\end{proof}

\section{Strong solvability of the Venttsel' problem}\label{sec3}
\setcounter{equation}{0}

\noindent We begin with the existence and uniqueness of the weak solution. By standard Galerkin methods (cf. \cite{ladyz}), the following result holds.

\begin{lemma}\label{weaksol} Let $f\in L^2(\Omega)$, $g\in L^2(\partial\Omega)$, $a\in L^\infty (Q_T)$ and $b\in L^\infty (\partial''Q_T)$. Then there exists a unique weak solution $u$ in $V^{1,0}(Q_T,\partial''Q_T)$ of problem \eqref{deb}. Moreover
\begin{equation}\label{stimacont}
\|u\|_{V^{1,0}(Q_T,\partial''Q_T)}\leq C(\|f\|_{L^2(Q_T)}+\|g\|_{L^2(\partial''Q_T)}),
\end{equation}
where $C$ depends on $T$, $a$ and $b$.
\end{lemma}

\noindent We finally prove the desired regularity for the weak solution of the parabolic nonlocal Venttsel' problem.

\begin{teo} Let $\sigma$ be subject to condition \eqref{boundsigma}. Suppose that $g$, $a$ and $b$ are as in Lemma \ref{weaksol} and that $f\in L^2_\sigma (\Omega)$. Then the problem \eqref{pbform1}-\eqref{pbform3} has a unique solution $u\in V^{2,1}_\sigma (Q_T,\partial''Q_T)$, and the following inequality holds:
\begin{equation}\label{regolarita}
\begin{split}
&\|u\|^2_{H^{1,0}(Q_T)}+\|r^\sigma D^2 u\|^2_{L^2(Q_T)}+\|r^\sigma u_t\|^2_{L^2(Q_T)}+\|u\|^2_{H^{2,1}(\partial''Q_T)}\\[3mm]
&\leq C\left(\|r^\sigma f\|^2_{L^2(Q_T)}+\|g\|^2_{L^2(\partial''Q_T)}\right),
\end{split}
\end{equation}
where $C$ depends on $\sigma$, $T$, $a$ and $b$.
\end{teo}

\begin{proof} We proceed similarly to the proof of \cite[Theorem 3.3]{nostronazarov}. We introduce the set of operators $\elle_\mu\colon V^{2,1}_\sigma(Q_T,\partial''Q_T)\to L^2_\sigma (Q_T)\times L^2(\partial''Q_T)$ as follows:
$$\elle_\mu u:=\left(u_t-\Delta u+\mu\,au,\left(u_t-\Delta_\ell u+\mu\left(\frac{\partial u}{\partial\nu}+bu+\theta_s(u)\right)\right)\Big|_{\partial\Omega}\right).$$

We claim that the operator $\elle_0$ is invertible. Indeed, it corresponds to the boundary value problem
$$u_t-\Delta u=f \quad\,\text{in}\,\,Q_T,\qquad u_t-\Delta_\ell u=g\quad\,\text{on}\,\,\partial''Q_T, \qquad u(x,0)=0\quad\,\text{on}\,\,\overline\Omega. $$
Here the equation in $Q_T$ and the equation on $\partial''Q_T$ are decoupled. So we can first solve the boundary equation and then use its solution as the Dirichlet datum for the equation in the domain. The estimates similar to Theorem \ref{aprioriest}, combined with Proposition \ref{tilde}, show that the solution belongs to $V^{2,1}_\sigma (Q_T,\partial''Q_T)$ and inequality \eqref{regolarita} holds. So the claim follows.

The estimates in Theorem \ref{aprioriest} show that the operator
$$\elle_\mu-\elle_0\colon V^{2,1}_\sigma(Q_T,\partial''Q_T)\to L^2_\sigma (Q_T)\times L^2(\partial''Q_T);\qquad\elle_\mu u-\elle_0 u=\mu\left(au,\frac{\partial u}{\partial\nu}+bu+\theta_s(u)\right)$$ is compact. Since ${\text{Ker}}(\elle_1)$ is trivial by Lemma \ref{weaksol}, the operator $\elle_1$ is also invertible, and the proof is complete.
\end{proof}

If $\Omega$ is a convex polygon, then $\alpha<\pi$. Hence, we can choose $\sigma=0$ and we obtain the following result.

\begin{corol} Let $\Omega$ be a convex polygon. Suppose that $f\in L^2 (Q_T)$, $g\in L^2(\partial''Q_T)$, $a\in L^\infty (Q_T)$ and $b\in L^\infty (\partial''Q_T)$. Then the problem \eqref{pbform1}-\eqref{pbform3} has a unique solution $u\in H^{2,1}(Q_T)\cap H^{2,1}(\partial''Q_T)$, and the following inequality holds:
\begin{equation}\notag
\|u\|^2_{H^{2,1}(Q_T)}+\|u\|^2_{H^{2,1}(\partial''Q_T)}\leq C(\|f\|^2_{L^2(Q_T)}+\|g\|^2_{L^2(\partial''Q_T)}),
\end{equation}
where $C$ depends on $T$, $a$ and $b$.
\end{corol}

If $\Omega$ is not convex, then $\pi<\alpha<2\pi$. In this case the solution in general does not belong to $H^{2,1}(Q_T)$ even for $f=0$, see e.g. \cite{kozlovmazya} for the asymptotics of solution to the Dirichlet problem.

\begin{oss} In \cite{nostronazarov} we considered the elliptic case, in particular we proved that the solution of the elliptic problem belongs to $H^2(\partial\Omega)$. Actually, similarly to Proposition \ref{tilde}, we have that $u\in\tilde H^2(\partial\Omega)$. In turn, this implies that the hypothesis $s<\frac{3}{4}$ is not needed in \cite[Theorem 2.1]{nostronazarov}.
\end{oss}

\begin{oss} All our results easily hold for an arbitrary piecewise smooth domain $\Omega\subset\R^2$ without cusps.
\end{oss}

\appendix
\section{Appendix. The extension theorem}\label{appendice}
\setcounter{equation}{0}

\begin{teo}\label{ext}
 Let $u\in H^{2,1}(\partial''Q_T)$ and $u|_{t=0}=0$. Then there exists an extension $U\in H^{\frac{5}{2},\frac{5}{4}}(\R^2\times\R)$ such that $U|_{t=0}=0$, $U|_{\partial'' Q_T}=u$, and 
 \begin{equation}\label{stimaUbis}
\|U\|_{H^{\frac{5}{2},\frac{5}{4}}(\R^2\times\R)}\leq C\|u\|_{H^{2,1}(\partial''Q_T)}.
\end{equation}

\end{teo}

\begin{proof}
By localization, we can consider separately the extension from a face $l_j\times (0,T)$ and the extension from a neighborhood of a corner. Using standard extension from a face to the plane containing this face
$$
H^{2,1}(l_j\times (0,T)) \to H^{2,1}(\R\times\R),
$$
we reduce the first operation to the extension from a plane and the second one to the extension from a pair of half-planes intersecting on the $t$-axis. Using a proper linear coordinates transform, we can assume that these half-planes are orthogonal. Since $u|_{t=0}=0$, we can suppose without loss of generality that the extended function is odd w.r.t. $t$. Moreover, in what follows all extensions are supposed compactly supported.
\medskip

We now denote 
$$
\gathered
\Pi^1=\{(x_1,x_2,t)\in \R^2\times\R: x_2=0\}; \quad \Pi^2=\{(x_1,x_2,t)\in \R^2\times\R: x_1=0\}; \\
\Pi^j_{\pm}=\{(x_1,x_2,t)\in \Pi^j: x_j\gtrless0\}, \quad j=1,2.
\endgathered
$$

We introduce a mollifier $\phi(x_1,t)\in {\cal C}^\infty_0(\R\times\R)$ such that $\phi$ is radially symmetric, $\support(\phi)\subset B_1(0)$, and $\displaystyle\int_{\R\times\R}\phi\, dx_1\, dt=1$. The extension from the plane $\Pi_1$ is defined in a standard way \cite{Slob} via the 2D Fourier transform (here $(\xi_1,\tau)$ are the variables dual to $(x_1,t)$):
\begin{equation}
\widehat U(\xi_1,x_2,\tau)=
(\widehat{{\cal P}_1u})(\xi_1,x_2,\tau):=\widehat \phi\big((1+\xi_1^4+\tau^2)^{\frac 14}\,|x_2|\big)\cdot \widehat u(\xi_1,\tau).
\end{equation}
Direct (and standard) calculation using the Parseval identity provides the estimate (\ref{stimaUbis}). Moreover, since $u$ is odd w.r.t. $t$, ${\cal P}_1u$ is also odd w.r.t. $t$. In the same way we define the extension operator ${\cal P}_2$ from the plane $\Pi_2$.

To manage the extension from $\Pi^1_+\cup\Pi^2_+$, we first extend 
$$
H^{2,1}(\Pi^1_+)\cap H^{2,1}(\Pi^2_+)\to H^{2,1}(\Pi^1)\cap H^{2,1}(\Pi^2) 
$$
and we apply the operator ${\cal P}_2$ to $u|_{\Pi^2}$. It remains to extend the function 
$$
v=\big(u-{\cal P}_2(u|_{\Pi^2})\big)_{\Pi_1}
$$
so that the extension $V$ vanishes on $\Pi^2_+$.

We split $v$ into the sum $v=v_0+v_++v_-$, where
$$
v_0(x_1,t)=\frac {v(x_1,t)-v(-x_1,t)}2\quad\text{ and }\quad v_{\pm}(x_1,t)=\frac {v(x_1,t)+v(-x_1,t)}2\cdot\chi(\pm x_1)
$$
($\chi$ stands for the Heaviside function).
Since $v_0$ is odd w.r.t. $x_1$, the function $V_0={\cal P}_1v_0$ is also odd w.r.t. $x_1$ and thus vanishes on $\Pi^2_+$.

Next, we notice that the function $v(x_1,t)+v(-x_1,t)$ is even w.r.t. $x_1$ and vanishes on the line $x_1=0$. Therefore, $v_{\pm}\in H^{2,1}(\Pi^1)$.

Since $v_+$ is supported in $\Pi^1_+$, we immediately obtain that the support of the function ${\cal P}_1v_+$ lies in the wedge $x_1\ge -|x_2|$. Hence, the function
$$
V_+(x_1,x_2,t)=({\cal P}_1v_+)(x_1-x_2, x_2,t)
$$
is an extension of $v_+$ having the required smoothness and vanishing on $\Pi^2_+$. In a similar way, we define 
$$
V_-(x_1,x_2,t)=({\cal P}_1v_-)(x_1+x_2, x_2,t).
$$
Setting $V=V_0+V_++V_-$, the thesis follows.
\end{proof}

\bigskip

\noindent {\bf Acknowledgements.} S. C. and M. R. L. have been supported by the Gruppo Nazionale per l'Analisi Matematica, la Probabilit\`a e le loro Applicazioni (GNAMPA) of the Istituto Nazionale di Alta Matematica (INdAM). A. N. was partially supported by Russian Foundation for Basic Research grant 18-01-00472. M. R. L. would like to acknowledge networking support by the COST
Action CA18232.

\noindent This work was undertaken while A. N. was visiting the department of Basic and Applied Sciences for Engineering of Sapienza Universit\`a di Roma supported by St. Petersburg University (project 41126689) and by 
the international agreement between Sapienza Universit\`a di Roma and Steklov Mathematical Institute of Russian Academy of Sciences. 

\noindent We are thankful to Professor Vladimir Gol'dshtein (Ben-Gurion University) for the hint to the proof of Theorem \ref{ext} and to Professor Mikhail Surnachev (Keldysh Institute of Applied Mathematics) for some important comments.

\medskip


\end{document}